\documentclass{amsart}

\usepackage[mathscr]{eucal}
\usepackage[ruled]{algorithm2e}
\usepackage{amscd}
\usepackage{amssymb}
\usepackage{xypic}
\usepackage{url}

\newcounter{bocal}

\newcommand{\res}{\operatorname{res}}

\newcommand{\isom}{\cong}

\newcommand{\tw}{\mathcal{TW}}

\def\ttmt(#1,#2,#3,#4)%
{\left(\begin{smallmatrix}
#1 & #2 \\
#3 & #4
\end{smallmatrix}\right)}
\newcommand{\pair}[1]{\langle #1 \rangle}
\renewcommand{\aa}{\alpha}
\newcommand{\bb}{\beta}

\newcommand{\p}{^\prime}

\newcommand{\abn}[3]{\begin{array}{l}
#1\\#2\\#3\end{array}}
\newcommand{\bbn}[3]{\begin{array}{c}
a= #1,~b= #2\\ \sigma=#3\end{array}}
\newcommand{\cbn}[3]{\begin{array}{l}
#2\\#3\end{array}}
\newcommand{\dbn}[3]{#3}
\newcommand{\ebn}[2]{\begin{array}{c}
#1\\ \sigma=#2\end{array}}

\newcommand{\Q}{\mathbb{Q}}

\renewcommand{\O}{\mathcal{O}}

\newcommand{\CC}{\mathbb{C}}

\newcommand{\PP}{\mathbb{P}}
\newcommand{\QQ}{\mathbb{Q}}
\newcommand{\RR}{\mathbb{R}}
\newcommand{\ZZ}{\mathbb{Z}}

\newcommand{\cO}{\mathcal{O}}
\newcommand{\scV}{\mathcal{V}}
\newcommand{\dtot}{{d/dt}}
\newcommand{\upH}{\mathcal{H}}
\newcommand{\Sh}{\mathcal{S}h}
\newcommand{\sltr}{\operatorname{SL}_2(\mathbb{R})}
\newcommand{\sltz}{\operatorname{SL}_2(\mathbb{Z})}
\newcommand{\vstar}[1]{V_{#1}^\ast}

\theoremstyle{plain}
\newtheorem{theorem}{Theorem}[section]
\newtheorem{proposition}[theorem]{Proposition}
\newtheorem{lemma}[theorem]{Lemma}

\theoremstyle{definition}
\newtheorem{definition}[theorem]{Definition}

\theoremstyle{remark}

\numberwithin{equation}{section}

\begin{document}

\title[Picard-Fuchs equations of families of QM abelian surfaces]
{Picard-Fuchs equations of families of QM abelian surfaces}

\author{Amnon Besser}
\address{School Of Mathematical And Statistical Sciences\\
Arizona State University\\
PO Box 871804, Tempe, AZ, 85287-1804\\
USA
}
\curraddr{
Department of Mathematics\\
Ben-Gurion University of the Negev\\
P.O.B. 653\\
Be'er-Sheva 84105\\
Israel
}

\author{Ron Livn\'e}
\address{Einstein Institute of Mathematics\\
Edmond J. Safra Campus, Givat Ram\\
The Hebrew University of Jerusalem\\
Jerusalem, 91904\\
Israel}

\begin{abstract}
We describe an algorithm for computing the Picard-Fuchs equation for a
family of twists of a fixed elliptic surface. We then apply this
algorithm to obtain the equation for several examples, which are
coming from families of Kummer surfaces over Shimura curves, as
studied in our previous work. We use this to find correspondenced
between the parameter spaces of our families and Shimura curves. These
correspondences can sometimes be proved rigorously.
\end{abstract}



\maketitle

\section{Introduction}
\label{sec:intro}

Let $\pi: X\to \PP^1$ be a family of complex algebraic varieties. As
$s\in \PP^1 $ varies, the periods of the fibers $X_s$, i.e., integrals
of holomorphically varying differential forms against a topologically
constant family of homology classes, satisfies a certain differential
equation, known as the \emph{Picard-Fuchs} equation, whose coefficients
are rational functions. These equations and their power series
solutions are interesting in several respects.

Many of the previously studied examples where of families of elliptic
curves with some extra structure. In hope of finding new applications,
we decided to study the related families associated with Shimura
curves.

Recall that if $B$ is a division quaternion algebra over the field
$\QQ$ of rational numbers, which is unramified at $\infty$ in the
sense that $B\otimes_\QQ \RR\isom \operatorname{M}_{2\times 2}
(\RR)$, and if $\mathcal{M}$ is a maximal order in $B$, the group
$\Gamma$ of norm one elements in $\mathcal{M}$ embedds in $\sltr$, via the
embedding of $B$ in $\operatorname{M}_{2\times 2} (\RR)$, as a
discrete subgroup. The quotients of the complex upper half plane
$\upH$ by $\Gamma$, and more generally by finite subgroups
$\Gamma'\subset \Gamma$, are algebraizable as moduli spaces of abelian
surfaces whose endomorphism algebras are certain orders in
$\mathcal{M}$ (so called QM abelian surfaces), together with some
extra structures. Under mild
assumptions they carry a universal family of such abelian
surfaces. Since K3 surfaces are easier to write down, it is natural to
consider instead the universal family of the associated Kummer
surfaces obtained by taking the quotient under multiplication by $\pm 1$. 

In trying to write explicit equations for QM Kummer surfaces, we were
led to study in~\cite{Bes-Liv10} families of quadratic twists of a fixed
elliptic surface (see Section~\ref{sec:elliptic}). We identified 11
examples as related to QM Kummer surfaces.
In each of these
examples we have a family of varieties over $\PP^1$ for which we know
that the generic fiber is isogenous, in a possibly complicated way, to
a Kummer surface associated with a QM abelian surface. Consequently,
there is a correspondence between the base spaces for these families,
and a Shimura curve. In~\cite{Bes-Liv10} we carried out a detailed
analysis of this correspondence, which was highly involved and required
a delicate study of finite discriminant forms.

The main contribution of the present work is Algorithm~\ref{mainalg}
which computes the Picard-Fuchs equation for a family of twists of a
fixed elliptic surface, together with Theorem~\ref{mainthm} that
claims its validity. A nice feature of our alogorithm is that it only
requires knowledge of the Picard-Fuchs equation of the elliptic
surface being twisted. As input to our algorithm we thus need a method
for computing the Picard-Fuchs equation for an elliptic surface. We
describe an algorithm for doing this borrowed from a MAPLE script of
F. Beukers. We prove that the algorithm works since we have found no
record of this in the literature.

After describing the algorithm and the proof of
the main theorem, we apply the algorithm to the study of the families
of QM Kummer surfaces described above. In all of the examples we
expect the resulting Picard-Fuchs equation to be of degree $3$ and to
further be a symmetric square (see Section~\ref{sec:k3}) of a degree
$2$ equation. This turns out indeed to be the case and we list the
resulting degree $2$ equations.

On the Shimra curve side, these degree $2$ equations have been studied
by Elkies in~\cite{Elk98}. This suggests a method of discovering and
verifying the correspondeces between the bases of the families we
study and the related Shimura curve, using these Picard-Fuchs
equations. We describe this in all the examples. We note that while
this method falls short of a rigorous proof of the existence of a
correspondence between the underlying moduli problems, it is far
easier than the analysis carried out in~\cite{Bes-Liv10}. In fact, we
needed the results of the present work to exclude some possibilities
in~\cite{Bes-Liv10} and furthermore, one of the cases there is left
unproved even though we know the correspondece using Picard-Fuchs
techniques.

\section{The Picard-Fuchs equation}
\label{sec:generalities}

We briefly recall the Picard-Fuchs differential equation for a family
of varieties over a curve. For further details see for example
\cite{Pet86}.

Let $C$ be a complex analytic curve and let $V/C$ be a local system of
$\CC$-vector spaces of dimension $n$ over $C$. The analytic rank $n$ vector
bundle $\scV := V\otimes_{\CC} \cO_C$
carries a canonical connection $\nabla$ defined by the condition that
it vanishes on sections of $V$. We fix a meromorphic vector field
$\dtot$ on $C$, e.g., the one associated with a rational parameter $t$
if $C=\PP^1$. Recall that $\dtot$ induces a \emph{covariant derivative}
operator
\begin{equation*}
  \nabla_\dtot : \scV \to \scV\;.
\end{equation*}

Let $\alpha$ be a meromorphic section of $\scV$. Since $\scV$ has rank
$n$, there is going to be a relation
\begin{equation*}
  \sum_{i=0}^m a_i (\nabla_\dtot)^i \alpha =0
\end{equation*}
with $m\le n$ and with $a_i$ meromorphic functions on $C$. We may
normalize this by insisting that $a_m=1$.

Suppose $\gamma \in V^\ast(U)$, for some open $U\in C$, where $V^\ast$
is the dual of $V$. The evaluation of $\gamma$ on $\alpha$, which we
suggestively write as $\int_\gamma \alpha$, is a meromorphic function
on $U$ and is called a \emph{period} of $\alpha$. Since $\nabla$ vanishes on
sections of $V$ it follows easily 
that the period $\int_\gamma \alpha$ satisfies the differential
equation
\begin{equation*}
  \frac{d^m}{dt^m} y + \sum_{i=0}^{m-1} a_i  \frac{d^i}{dt^i} y =0\;,
\end{equation*}
which is called the \emph{Picard-Fuchs} equation associated with
$\alpha$.

When $V$ comes from geometry, a bit more can be said. Suppose that
$\pi: X\to C$ is a smooth projective family of algebraic varieties, and that
$V$ is the family of cohomology groups
\begin{equation*}
  V = \RR^l \pi_\ast \CC
\end{equation*}
For some non-negative integer $l$. In this case, $\scV$ is canonically identifies  with the vector bundle
of \emph{de Rham} cohomology groups,
\begin{equation*}
  \scV \isom \RR^l \pi_\ast \Omega_{X/C}^\bullet \;,
\end{equation*}
and the connection $\nabla$ is identified with the \emph{Gauss-Manin}
connection on the latter vector bundle. If $C$ and $\pi$ are algebraic, it
follows easily that we may take $\alpha$ to be an algebraic
(meromorphic) section of $\scV$ and that then the coefficients $a_i$
in the Picard-Fuchs equation will be rational functions on $C$. We
will call this a Picard-Fuchs equation associated with the $H^l $ of
the family.

In geometric situations we may futher take $\alpha$ to be a section in
the sub-bundle $\pi_\ast \Omega_{X/C}^l$ and we may take $\gamma$ to
be a family of Homology classes, so that the associated period is now
indeed the integral $\int_\gamma \alpha$.

In applications, it will always be the case that the sub-bundle
$\pi_\ast \Omega_{X/C}^l$ will be of rank $1$. Thus, $\alpha$ is
determined up to a product by a rational function. The Picard-Fuchs
equation is in some sense unique then, since we may recover easily the
equation associated with such a product from the equation for $\alpha$
(see also Section~\ref{sec:schwartzian} for how to remove the
remaining ambiguity).

We can also consider Picard-Fuchs equations associated with sub-local
systems $V\subset \RR^l \pi_\ast \CC$ provided our chosen $\alpha$
resides in $V\otimes \cO_C$.

We now list the local systems considered in this work.
Let $\upH$ be the complex upper half plane. We have a family of
elliptic curves $\pi^u: E^u\to \upH$ where for $\tau\in \upH$ we have
\begin{equation*}
  E_\tau^u = \CC/\ZZ\langle 1,\tau\rangle\;.
\end{equation*}
We consider the resulting local system
\begin{equation*}
  \Sh := \RR^1 \pi_\ast \CC\;,
\end{equation*}
which has a constand fiber $\CC^2$.
See~\cite[\S~12]{Zuck79} for a detailed discussion.
Note that $\pi_\ast^u \Omega_{E^u/\upH}^1 $ has the section $dz$,
where $z$ is the standard coordinate on $\CC$,
whose associated periods are $1$ and $\tau$, hence its Picard-Fuchs
equation is $y''=0$.

Let $\Gamma\subset \sltr $ be a discrete group, acting on $\upH$ by fractional
linear transformations. It acts on $\Sh$ in the via the standard
representation of $\sltr$ on $\CC^2$.
When $\Gamma\subset \sltz$ is a congruence subgroup, the quotient
$X_\Gamma :=\Gamma \backslash \upH$ has a
family $E_\Gamma :=\Gamma \backslash E^u$ of elliptic curves above it,
and both are algebraizable. The quotient $\Gamma \backslash \Sh$ is a
local system on $X_\Gamma$, isomorphic to $ \RR^1 \pi_\ast^\Gamma \CC
$, with $\pi^\Gamma$ the induced projection.

\newcommand{\Symm}{\operatorname{Symm}}
Let $B$ be an indefinite rational quaternion algebra and let $\Gamma\ \subset B^\times $ be as in
the introduction. Let $\pi^u: A^u\to X_\Gamma$
be the associated universal
family of abelian surfaces with quaternionic multiplication. Then
(see~\cite{Bes95}) the local system $\RR^2 \pi_\ast^u \CC $ splits as
a direct sum of a $3$-dimensional local system and a system isomorphic
to the symmetric square of $\Sh$, $\Symm^2(\Sh) $.

\newcommand{\Kum}{\operatorname{Kummer}}
For any family $\pi^A: A\to X$ of abelian surfaces, let $\pi^S:
S=\Kum(A)\to X $ be the 
associated family of Kummer surfaces. Then, the local system
$\RR^2 \pi_\ast^S \CC $ splits as a sum of a $16$-dimensional trivial
system and $\RR^2 \pi_\ast^A \CC $. 
In particular, when $A=A^u$ we see that $\RR^2 \pi_\ast^S \CC $ splits
as a sum of a $19$-dimensional trivial system and $\Symm^2(\Sh)
$. Furthermore, we have $\pi_\ast^S \Omega_{S/X}^2 \subset
\Symm^2(\Sh)\otimes \O_{X_\Gamma} $. Consequently the Picard-Fuchs
equation satisfied by the periods of a relative $2$-form $\omega $ on $S$ is
going to be of degree $3$, and will be the symmetric square of a
Picard-Fuchs equation of degree $2$ associated with the local system
$\Sh$ (see Section~\ref{sec:k3} for symmetric squares of equations).

\section{Elliptic surfaces and their Picard-Fuchs equations}
\label{sec:elliptic}

 An elliptic
surface, always considered over $\PP^1$, is a
smooth and connected compact complex algebraic
surface $E$, together with a surjective morphism
$\pi: E\to \PP^1$, such that the generic fiber is
a curve of genus $1$. We will always assume that
the fibration is relatively minimal and has a
given section, denoted $0$.

For all but a finite number of points $s\in
\PP^1$, the fiber $E_s=\pi^{-1}(s)$ is an
elliptic curve. The singular locus $\Sigma =
\Sigma(E)$ of the fibration is the (finite)
subset of $\PP^1$ over which the fibers are
singular (namely $\pi$ is not everywhere smooth).
Kodaira~\cite{Kod62} classified all possible types of
singular fibers (see also \cite[Chapter
V.7]{BHPV04}).

The generic fiber of an elliptic surface may be given by a Weierstrass
equation of the form $y^2=f(x)$, where $f(x)=ax^3+bx^2+cx+d$ and
$a,b,c,d$ are rational functions of the parameter $t$ on $\PP^1$. 

 Given two
distinct points $\alpha$ and $\beta$ in $\PP^1$, the
quadratic twist $E_{\alpha,\beta}$ at these points can be
described in two ways. Algebraically, if $E$ has
Weierstrass equation $y^2=f(x)$ and $\alpha$ and $\beta$
are finite points, then $E_{\alpha,\beta}$ has equation
\begin{equation}\label{twisteq}
  \frac{t-\alpha}{t-\beta}y^2= f(x)\;.
\end{equation}
Analytically, $E_{\alpha,\beta}$ can be described as
follows. Take the double cover $B'\to \PP^1$
ramified at $\alpha$ and $\beta$ and let $E\p$ be the
pullback surface. Now quotient $E\p$ by the
transformation which identifies the two fibers
above each fiber of $E$ with sign $-1$.

\begin{definition}\label{famtwist}
Let $E\to \PP^1$ be an elliptic surface as above
with $s\in\Sigma=\Sigma(E)$. For
$\lambda\in\PP^1-\Sigma$ let $E_{s,\lambda}$
be the twisted family at $s$ and at $\lambda$.
These surfaces vary in a family $\tw_{s}(E)$
over  the \emph{$\lambda$-line}\/ $\PP^1(\lambda)-\Sigma$.
\end{definition}

The local system $\RR^1 \pi_\ast \CC $ over $\PP^1-\Sigma$ has
dimension $2$. Its dual is the \emph{homological invariant} (tensored
with $\CC$) associated
by Kodaira to the elliptic surface, and we denote it by $F$.

A Picard-Fuchs equation for the $H^1$ of a general elliptic sufrace
$E$, corresponding to the invariant differential $\omega = dx/y $, can be
computed using Algorithm~\ref{ellipticpf}. It is taken from a MAPLE
script of F. Beukers (see Section~\ref{sec:software}). We failed to find it documented anywhere so we
give a short proof that it indeed works.

\begin{algorithm}\label{ellipticpf}
  \SetKwFunction{SOLVE}{SOLVE}
  \SetKwFunction{COEFFICIENTS}{COEFFICIENTS}
  \KwIn{An elliptic surface given by a Weierstrass equation
    $y^2=ax^3+bx^2+cx+d $, with $a,b,c,d$ rational functions of $t$}
  \KwOut{The Picard Fuchs equation $y''+ c_1 y' +c_2 y = 0$ satisfied
    by the periods of the invariant differential $\omega= dx/y$}
  \BlankLine
  $f \leftarrow ax^3+bx^2+cx+d $\;
  $f_t \leftarrow \frac{\partial f}{\partial t}$\;
  $f_{tt} \leftarrow \frac{\partial f_t}{\partial t}$\;
  $f_x \leftarrow \frac{\partial f}{\partial x}$\;
  $q \leftarrow q_4 x^4 + q_3 x^3 + q_2 x^2 + q_1 x+q_0$\;
  $q_x \leftarrow \frac{\partial q}{\partial x}$\;
  $e \leftarrow \frac{-f_{tt} \cdot f}{2}+ \frac{3 f_t^2}{4} - c_1
    \frac{f_t \cdot f}{2}+c_2 f^2+ \frac{3 f_x \cdot q}{2}-f\cdot q_x$\;
  $C \leftarrow $ \COEFFICIENTS{$e,x$}\;
  $(c_1,c_2,q_0,q_1,q_2,q_3,q_4) \leftarrow$ \SOLVE{$C=0$}\;
  \caption{Computing a Picard-Fuchs equation for an elliptic surface}
\end{algorithm}
\begin{proposition}
  Algortithm~\ref{ellipticpf} gives the Picard-Fuchs equation for the
  elliptic surface $E$.
\end{proposition}
\begin{proof}
We express $y$ in terms of $x$ as $y= f(x)^{1/2}$. Applying the
covariant Gauss-Manin differentiation with respect to $t$ amounts to
differentiating (after eliminating $y$)
with respect to $t$. On the invariant differential $\omega = f(x)^{-1/2}
dx$ we find
\begin{align*}
  \nabla_\dtot \omega &= -\frac{1}{2} f^{-\frac{3}{2}} f_t dx\\
  \nabla_{\dtot}^2 \omega &= \left(\frac{3}{4} f^{-\frac{5}{2}} f_{t}^2
    -\frac{1}{2} f^{-\frac{3}{2}} f_{tt}\right) dx\;.
\end{align*}
Now we may write the general differential operator of degree $2$
applied to $\omega$,
\begin{equation}\label{pfside1}
  \nabla_{\dtot}^2 \omega + c_1 \nabla_\dtot \omega + c_2 \omega
 =  \left(\frac{3}{4} f^{-\frac{5}{2}} f_{t}^2
    -\frac{1}{2} f^{-\frac{3}{2}} f_{tt}  -\frac{1}{2}c_1
    f^{-\frac{3}{2}} f_t + c_2   f^{-\frac{1}{2}} \right)dx\;. 
\end{equation}
For the appropriately chosen $c_1$ and $c_2$ this will give a trivial
de Rham cohomology class on $E/\PP^1$.
Reduction theory (see for
example~\cite{Ked01}) tells us that it is going to be the
differential of a rational function of the form $q(x)/y^n$ and
examining the poles at the $2$-torsion points shows that one can take
$n=3$ and $q$ a
polynomial of degree at most $4$. This is given by 
\begin{equation}\label{pfside2}
  d \frac{q(x)}{y^3} =  d \frac{q(x)}{f^{3/2}} = \left(-\frac{3}{2} q \cdot
  f^{-\frac{5}{2}} \cdot f_x + q_x   f^{-\frac{3}{2}}\right) dx\;
\end{equation}
To find the Picard-Fuchs equation, we equate~\eqref{pfside1}
to~\eqref{pfside2}, multiply by $f^{5/2}$ to clear denominators. This
gives the quantity $e$ in the algorithm. Then we simply solve $e=0$,
identically in $x$, expressing the $c$'s and $q$'s in terms of $t$.
\end{proof}

\section{The Picard-Fuchs equation for a family of twists}
\label{sec:twists}

In this section we prove our main theorem, Theorem~\ref{mainthm},
which describes the differential equation satisfied by the periods of
the $H^2$ of the family of
twists $\tw_{s}(E)$, described in Definition~\ref{famtwist}, of a
fixed elliptic surface $E$. We will in fact show  
that the periods for this $H^2$ can be
recovered from the periods for $H^1$ of $E$ and the
differential equation can be recovered solely based on a
Picard-Fuchs equation for $H^1$ of $E$.

To simplify the notation, we assume that $s=0$, i.e., that the twists
are at $0$ and a varying point.
Recall from the description following~\eqref{twisteq} that $E_{0,\lambda}$ can be obtained from $E$ as follows: One takes a
double covering $\pi_\lambda: B'\to \PP^1$ which is ramified exactly 
over $0$ and $\lambda$. Let $d_\lambda:B'\to B'$ be the deck
trasformation of the covering. One considers the pullback
$\pi_\lambda^\ast E$ and takes the quotient $\pi_\lambda^\ast E/D_\lambda$ 
where $D_\lambda$ is the map $(s,e)\mapsto (d_\lambda(s),-e)$, i.e., the map
that identifies the fibers at $s$ and $d_\lambda(s)$ but via the map $-1$.
The result may have singularities at the fixed points $0$ and $\lambda$ of
$d_\lambda$ and resolving them one obtains $E_{0,\lambda}$. We
henceforth ease notation and write $E_\lambda$ for $E_{0,\lambda}$.

We now write a homology class $\Gamma_\lambda \in H_2(E_\lambda,\CC)
$. We will obtain $\Gamma_\lambda $ by modifying a fixed homology
class $\Gamma'\in H_2(E,\CC) $. In fact, we take $\Gamma'$ in
$H_1(\PP^1-\Sigma,F)$, where $F$ is the homological invariant (see
Section~\ref{sec:elliptic}).
An element of $H_1(\PP^1-\Sigma,F) $
consists of a formal sum $\sum (\gamma_i,x_i)$ where $\gamma_i$ are paths
in $\PP^1-\Sigma$ and $x_i$ is a section of $F_{\gamma_i}$, in such a
way that the obvious boundary map vanishes.  Write the one
form on the elliptic surface $E$, $dx/y$, as a family of differential
forms $\omega_t$ and consider the function $G_i$ on the path
$\gamma_i$ given at a point $t$ by $G_i(t) = \int_{x_i(t)} \omega_t$.

Having fixed $\Gamma'$ we can write a family of $2$ homology
classes $\Gamma_\lambda \in H_2(E_\lambda,\CC)$ as follows. Each path
$\gamma_i$ can
be pulled back to
$B'$. If one of the pullbacks is $\delta_i$ then the other one is
$d_\lambda(\delta_i) $. The section
$x_i$ pull back to both of these
lifts. Since we identify
the fiber at $s$ with the fiber at $d_\lambda(s)$ via the map $-1$ 
and since $-1$ acts as $-1$ on the first homology it is
clear that $\Gamma_\lambda^\prime :=\sum (\delta_i,x_i)+
\sum (d_\lambda(\delta_i),-x_i)$ descents to the required  homology
class $\Gamma_\lambda \in H_2(E_\lambda,\CC)$.

Let us now write explicitely a period associated with this homology
class. We first need to write a holomorphic
differential two form $\eta_\lambda$ on $E_\lambda$. We have a two form on $E$, $\eta = \omega_t \wedge dt$.
We can write an affine model for $B'$, the double cover of $\PP^1$
ramified at $0$ and $\lambda$, as $s^2=t(t-\lambda)$. The form $s^{-1}\cdot
\pi_\lambda^\ast $ has the right behavior with respect to deck
transformations and therefore descents to the required form
$\eta_\lambda$ on $E_\lambda$. Note that the
choice of $s^{-1}$ as a multiplier eliminates the zeros that $dt$
acquires from the ramified cover. We now compute the period
$\int_{\Gamma_\lambda}\eta_\lambda$. Let us write the function
obtained by evaluating An easy computation shows this is
equal to
\begin{align*}
  &\phantom{=} \int_{\Gamma_\lambda^\prime} s^{-1} \pi_\lambda^\ast
  \eta \\
  &= \sum_i \left( \int_{\delta_i} s^{-1} G_i(\pi_\lambda(s))
    \pi_\lambda^\ast dt + \int_{d_\lambda \delta_i} s^{-1} (-G_i(\pi_\lambda(s)))
    \pi_\lambda^\ast dt\right)\\
   &= 2 \sum_i  \int_{\delta_i} s^{-1} G_i(\pi_\lambda(s))
    \pi_\lambda^\ast dt \\
    &= 2 \sum_i \int_{\gamma_i} (t(t-\lambda))^{-1/2} G_i(t) dt\;.
\end{align*}
Dividing by $2$ we get the period 
\[
  u(\lambda):=\sum \int_{\gamma_i} (t(t-\lambda))^{-1/2} G_i(t) dt
\]

Our goal is now to compute a differential equation satisfied by
$u$. In doing so, we will only use the fact that the $G_i$ satisfy the
Picard-Fuchs equation for the elliptic family $E$, which we recalled
in section~\ref{sec:elliptic},
\begin{equation}\label{eq:ellipticpf}
  y''+c_1(t)y'+c_2(t)y = 0\; .
\end{equation}

The computation is inspired
by the computation in~\cite[2.10]{Clem03}.

\begin{lemma}
  Suppose $y$ satisfies~\eqref{eq:ellipticpf}. Then, for a fixed
  $\lambda$, the function $z=z_\lambda=
  (t(t-\lambda))^{-1/2} y$ satisfies the equation
  \begin{equation*}
    z''+\aa_\lambda(t)z'+\bb_\lambda (t)z = 0
  \end{equation*}
  with
  \begin{align*}
    \aa_\lambda(t)&= c_1(t)+\frac{2t-\lambda}{t(t-\lambda)}\\
    \bb_\lambda(t)&= c_2(t)+c_1(t)\frac{2t-\lambda}{2t(t-\lambda)}-
    \frac{\lambda^2}{4t^2(t-\lambda)^2}
  \end{align*}
\end{lemma}
\begin{proof}
A straightforward computation.
\end{proof}

Suppose now that we are given two rational functions $p(t)=p_\lambda(t)$ and
$q(t)=q_\lambda(t)$. We have
\begin{equation*}
  (pz+qz')' = p'z+(p+q')z'+qz''
\end{equation*}
If we force the relation

\begin{equation}
  p+q'=\aa_\lambda q\label{eq:keyrel}
\end{equation}
then we can write
\begin{equation*}
   p'z+(p+q')z'+qz'' = p'z+q(z''+\aa_\lambda z')= p'z-q\bb_\lambda z\;,
\end{equation*}
by using the differential equation for $z$. The
relation~\eqref{eq:keyrel} gives $p=\aa q-q'$
so that $p'=\aa'q+q'\aa -q''$ and we finally end up with the relation
\begin{equation*}
   (pz+qz')' = (\aa'q+q'\aa -q''-q\bb )z\;.
\end{equation*}
Now, we can do the following: We have $u(\lambda) = \sum
\int_{\gamma_i} z_\lambda(t) dt$. Since
$z$ depends on $\lambda$ only through division by $\sqrt{t-\lambda} $,
we easily get by differentiating $n$ times with respect to $\lambda$ below the
integral sign,
\begin{equation}\label{lamder}
 \frac{d^n u}{d\lambda^n}= \frac{1\cdot 3\cdot\cdots \cdot
   (2n-1)}{2^n} \sum_i \int_{\gamma_i} \frac{z}{(t-\lambda)^n} dt
\end{equation}
\begin{lemma}
  There is a choice for $q$
  such that  we may expand $ \aa'q+q'\aa -q''-q\bb $ as a polynomial in
$(t-\lambda)^{-1}$ with coefficients which are rational functions of  $\lambda$,
\begin{equation}\label{eq:expand}
   \aa'q+q'\aa -q''-q\bb = \sum_n c_n(\lambda) (t-\lambda)^{-n}
\end{equation}
\end{lemma}
\begin{proof}
First let $q_0$ be the least common multiple of the denominators of
$\aa$ and $\bb$ as rational functions of $t$. Then, $ \aa'q_0+q_0^\prime \aa
-q_0^{\prime\prime}-q_0 \bb $ is a polynomial in $t$ and can therefore
also be written as a polynomial in $t-\lambda$, with coefficients
which are rational functions of $\lambda$. Suppose that this
polynomial has degree $m$. Then, we may simply take
$q=q_0 (t-\lambda)^{-m}$.
\end{proof}
We may modify the paths $\gamma_i$ to homotopic paths making sure that
$\sqrt{t-\lambda}$ is single valued on each path. Also, the sums of
the monodromies of the $G_i$ around the paths $\gamma_i$ is $0$
because $\Gamma'$ is closed. Thus, we have,
\begin{align*}
  0 &= \sum_i \int_{\gamma_i}  \frac{d}{dt}\left(p_\lambda z_\lambda +q_\lambda
  \frac{d}{dt}z_\lambda \right) dt \\
  &= \sum_i \int_{\gamma_i}  \left(\frac{d\aa_\lambda}{dt} q_\lambda +
    \frac{d q_\lambda }{dt}\aa_\lambda  - \frac{d^2 q_\lambda}{dt^2}
    -q_\lambda \bb_\lambda \right)z_\lambda dt \\
  &= \sum_i \int_{\gamma_i}  \sum_n c_n(\lambda) (t-\lambda)^{-n} z_\lambda dt \\
  &=  \sum_n  c_n(\lambda) \sum_i \int_{\gamma_i}   \frac{ z_\lambda}{(t-\lambda)^n} dt \\
  &= \sum_n \tilde{c}_n(\lambda)  \frac{d^n u}{d\lambda^n}\;,
\end{align*}
by~\eqref{lamder}, with
$$ \tilde{c}_n(\lambda) =\frac{2^n}{1\cdot 3\cdot\cdots \cdot
  (2n-1)} c_n(\lambda)\;.$$
We therefore proved the following.
\begin{theorem}\label{mainthm}
  Let $E$ be an elliptic surface whose periods satisfy the
  differential equation~\eqref{eq:ellipticpf}. Then, algorithm~\ref{mainalg}
  computes a
  differential equation with polynomial coefficients satisfied by a
  non-trivial period for $H^2$ of the family $\tw_{0}(E)$.
\end{theorem}

\begin{algorithm}\label{mainalg}
  \SetKwFunction{LCM}{LCM}
  \SetKwFunction{DENOMINATOR}{DENOMINATOR}
  \SetKwFunction{DEG}{DEG}
  \SetKwFunction{COEFFICIENT}{COEFFICIENT}
  \KwIn{A Picard-Fuchs equation $y''+c_1(t)y'+c_2(t)y = 0$ for an
    elliptic surface $E$}
  \KwOut{A vector $\tilde{c}$ such that a Picard-Fuchs equation for
    the family of twists $\tw_{0}(E)$ is given by $\sum_n
    \tilde{c}_n(\lambda)  \frac{d^n u}{d\lambda^n} $}
  \BlankLine
  $\aa\leftarrow  c_1(t)+\frac{2t-\lambda}{t(t-\lambda)}$\;
  $\bb\leftarrow  c_2(t)+c_1(t)\frac{2t-\lambda}{2t(t-\lambda)}-
    \frac{\lambda^2}{4t^2(t-\lambda)^2} $\;
  $q_0 \leftarrow$ \LCM{\DENOMINATOR{$\aa$},\DENOMINATOR{$\bb$}}\;
  $pol_0 \leftarrow  \aa'q_0+q_0^\prime \aa
  -q_0^{\prime\prime}-q_0 \bb$\;
  $m \leftarrow$ \DEG{$pol_0$}\;
  $q\leftarrow q_0 (t-\lambda)^{-m}$\;
   $pol \leftarrow  (\aa'q+q'\aa -q''-q\bb)_{t\leftarrow s+\lambda}$\; 
  $c_n(\lambda) \leftarrow $ \COEFFICIENT{$s^{-n}$ in $pol$}\;
  $\tilde{c}_n(\lambda) \leftarrow \frac{2^n}{1\cdot 3\cdot\cdots \cdot
  (2n-1)} c_n(\lambda)$\;
  \caption{Computing a differential equation for periods of
    $\tw_{0}(E)$ }
\end{algorithm}

\section{K3 surfaces}
\label{sec:k3}

In~\cite{Bes-Liv10} we studied a particluar class of elliptic fibrations. Out
of the list of elliptic fibrations with $4$ singular fibers compiled
by Herfurtner~\cite{Her91}, we picked out the ones for which the twists
give K3 surfaces, generically with Picard number $19$. These K3
surfaces are then isogenous to Kummer surfaces associted with Abelian
surfaces whose isogeny algebra is a rational quaternion algebra. We further
picked out only the examples in which the quaternion algebra in
question is indefinite There are
11 examples, which we list below 
(Table~1, see also~\cite[Table~1]{Bes-Liv10}). The
method for deciding which families
of twists correspond to quaternion algebras and the method for
determining the discriminant of the associated algebra are detailed
in~\cite[Proposition~2.4.1 and Lemma~2.5.1]{Bes-Liv10}.

As discussed in Section~\ref{sec:generalities}, for each of the
examples above, the resulting Picard-Fuchs equation should be of
degree $3$ and should be a symmetric square of a degree $2$ equation,
which is a Picard-Fuchs equation for the Shimura local system
descended to the base. In this section we verify that this is indeed the
case, and we compute the degree $2$ equations.

Symmetric squares of differential equations are considered, for
example in~\cite[Example~6.5.2]{Pet86}. Given a differential equation of degree
$2$, $y''+ay'+by=0$, one looks for the equations satisfied by $z=y^2$. The result
is
\begin{equation}
  \label{eq:symmsquare}
  z'''+\alpha z''+ \beta z' + \gamma z = 0 \text{ with } \alpha=3a,\;
  \beta= 4b+2 a^2 + a',\; \gamma= 4ab+2b'\;.
\end{equation}
If we are given a differential equation of degree 3,  we can
check if it is a symmetric square of one of degree 2 and find the
``square root'' as described in Algorithm~\ref{squareroot}:

\begin{algorithm}\label{squareroot}
  \KwIn{A differential equation $ z'''+\alpha z''+ \beta z' + \gamma z = 0$}
  \KwOut{A differential equation $y''+ay'+by=0$ whose symmetric square
    equals the input equation, if it exists}
  \BlankLine
  $a \leftarrow \alpha/3$\;
  $b \leftarrow (\beta-2 a^2-a')/4$\;
  $c \leftarrow \gamma-4ab-2b'$\;
  \caption{Taking the square root of a degree 3 differential equation}
\end{algorithm}

Not surprisingly, in all 11 examples, the resulting differential
equation turn out to be a symmetric square of an equation of degree
2. In table~\ref{bigtablenew} below we list the examples together with the
resulting equations of degree 2 (one can recover the degree 3 equation
using~\eqref{eq:symmsquare}).

The first column is the example number, which is the same
as in Table~\ref{bigtable} in~\cite{Bes-Liv10}. The second column gives the
types of singular fibers for the based elliptic surface and their
locations and the third column gives the coefficients of the degree
$2$ equation. The final column gives the expected discriminant for the
associated quaternion algebra. In the table $\gamma = \frac{-{\left( 1 +
{\sqrt{-2}} \right) }^4}{3}$ and $\delta =
\frac{{\left( 1 + {\sqrt{-7}} \right)
}^7}{512}$. Conjugates for such elements are
over $\Q$.

\begin{table}[htbp]
\[
\begin{array}{||l||l|l|l||}

  1& \abn{ I_1 ,I_1 ,I_8 ,II }
  {\gamma,\bar{\gamma},\infty,0}
    { \frac{3}{4},\frac{3}{4},\frac{3}{4},\frac{35}{36}}&
    \bbn{ \frac{27 - 21\lambda  +
       6{\lambda }^2}{27\lambda  - 14{\lambda }^2 +
       3{\lambda }^3}}
    {\frac{3\left( -1 - 6\lambda  + 3{\lambda }^2
         \right) }{16{\lambda }^2
       \left( 27 - 14\lambda  + 3{\lambda }^2 \right) }}
    {\frac{3\left( 945 - 652\lambda  +
         142{\lambda }^2 - 60{\lambda }^3 +
         9{\lambda }^4 \right) }{4{\lambda }^2
       {\left( 27 - 14\lambda  + 3{\lambda }^2 \right)
           }^2}}&6\\\hline
  2& \abn{ I_1 ,I_2 ,I_7 ,II }
  {\frac{-9}{4},\frac{-8}{9},\infty,0}
   { \frac{3}{4},\frac{3}{4},\frac{3}{4},\frac{35}{36}}&
   \bbn{ \frac{144 + 339\lambda  + 144{\lambda }^2}
     {144\lambda  + 226{\lambda }^2 + 72{\lambda }^3}}
    {\frac{-2 + 36\lambda  + 27{\lambda }^2}
     {4{\lambda }^2\left( 72 + 113\lambda  +
         36{\lambda }^2 \right) }}
    {\frac{20160 + 42008\lambda  + 41331{\lambda }^2 +
       17388{\lambda }^3 + 3888{\lambda }^4}{4
       {\lambda }^2{\left( 72 + 113\lambda  +
           36{\lambda }^2 \right) }^2}}&6\\\hline
  3& \abn{ I_1 ,I_4 ,I_5 ,II }{ -10,0,\infty,\frac{1}{8}}
   { \frac{3}{4},\frac{3}{4},\frac{3}{4},\frac{35}{36}}&
   \bbn{ \frac{-5 + 119\lambda  + 16{\lambda }^2}
     {\lambda \left( -10 + 79\lambda  +
         8{\lambda }^2 \right) }}
    {\frac{6\left( -1 + 7\lambda  + 2{\lambda }^2
         \right) }{{\left( 1 - 8\lambda  \right) }^2
       \lambda \left( 10 + \lambda  \right) }}
    {\frac{3\left( 25 - 210\lambda  +
         2179{\lambda }^2 + 216{\lambda }^3 +
         16{\lambda }^4 \right) }{{\left( 1 -
           8\lambda  \right) }^2{\lambda }^2
       {\left( 10 + \lambda  \right) }^2}}&15\\\hline
  4& \abn{ I_2 ,I_3 ,I_5 ,II }{\frac{-5}{9},0,\infty,3}
   { \frac{3}{4},\frac{3}{4},\frac{3}{4},\frac{35}{36}}&
   \bbn{ \frac{15 + 39\lambda  - 36{\lambda }^2}
     {30\lambda  + 44{\lambda }^2 - 18{\lambda }^3}}
    {\frac{-23 - 246\lambda  + 81{\lambda }^2}
     {48{\left( -3 + \lambda  \right) }^2\lambda
       \left( 5 + 9\lambda  \right) }}
    {\frac{2025 + 4295\lambda  + 9156{\lambda }^2 -
       1809{\lambda }^3 + 729{\lambda }^4}{12
       {\left( -3 + \lambda  \right) }^2{\lambda }^2
       {\left( 5 + 9\lambda  \right) }^2}}&10\\\hline
  5& \abn{ I_1 ,I_1 ,I_7 ,III }
  {\delta,\bar{\delta},\infty,0}
  { \frac{3}{4},\frac{3}{4},\frac{3}{4},\frac{15}{16}}&
   \bbn{ \frac{64 + 39\lambda  + 16{\lambda }^2}
     {64\lambda  + 26{\lambda }^2 + 8{\lambda }^3}}
    {\frac{-2 + 4\lambda  + 3{\lambda }^2}
     {4{\lambda }^2\left( 32 + 13\lambda  +
         4{\lambda }^2 \right) }}
    {\frac{3840 + 2072\lambda  + 43{\lambda }^2 +
       220{\lambda }^3 + 48{\lambda }^4}{4
       {\lambda }^2{\left( 32 + 13\lambda  +
           4{\lambda }^2 \right) }^2}}&14\\\hline
  6& \abn{ I_1 ,I_2 ,I_6 ,III }{ 4,1,\infty,0}
  { \frac{3}{4},\frac{3}{4},\frac{3}{4},\frac{15}{16}}&
   \bbn{ \frac{8 - 15\lambda  + 4{\lambda }^2}
     {2\lambda \left( 4 - 5\lambda  +
         {\lambda }^2 \right) }}
    {\frac{-1 - 6\lambda  + 3{\lambda }^2}
     {16{\lambda }^2
       \left( 4 - 5\lambda  + {\lambda }^2 \right) }}
    {\frac{3\left( 20 - 33\lambda  + 28{\lambda }^2 -
         7{\lambda }^3 + {\lambda }^4 \right) }{4
       {\lambda }^2{\left( 4 - 5\lambda  +
           {\lambda }^2 \right) }^2}}&6\\\hline
  7& \abn{ I_1 ,I_3 ,I_5 ,III }
  {\frac{-25}{3},0,\infty,\frac{1}{5}}
   { \frac{3}{4},\frac{3}{4},\frac{3}{4},\frac{15}{16}}&
   \bbn{ \frac{25 - 369\lambda  - 60{\lambda }^2}
     {50\lambda  - 244{\lambda }^2 - 30{\lambda }^3}}
    {\frac{-167 + 630\lambda  + 225{\lambda }^2}
     {16{\left( 1 - 5\lambda  \right) }^2\lambda
       \left( 25 + 3\lambda  \right) }}
    {\frac{15\left( 125 - 675\lambda  +
         4244{\lambda }^2 + 501{\lambda }^3 +
         45{\lambda }^4 \right) }{4
       {\left( 1 - 5\lambda  \right) }^2{\lambda }^2
       {\left( 25 + 3\lambda  \right) }^2}}&6\\\hline
  8& \abn{ I_2 ,I_3 ,I_4 ,III }{ \frac{-1}{3} ,0,\infty,1}
   { \frac{3}{4},\frac{3}{4},\frac{3}{4},\frac{15}{16}}&
   \bbn{ \frac{1 + 3\lambda  - 12{\lambda }^2}
     {2\lambda  + 4{\lambda }^2 - 6{\lambda }^3}}
    {\frac{-1 - 9\lambda  + 9{\lambda }^2}
     {16{\left( -1 + \lambda  \right) }^2
       \left( \lambda  + 3{\lambda }^2 \right) }}
    {\frac{3\left( 1 + 3\lambda  + 13{\lambda }^2 -
         6{\lambda }^3 + 9{\lambda }^4 \right) }{4
       {\left( -1 + \lambda  \right) }^2
       {\left( \lambda  + 3{\lambda }^2 \right) }^2}}
       &6\\\hline
   9& \abn{ I_1 ,I_1 ,I_6 ,IV }{ 1,-1,\infty,0}
   { \frac{3}{4},\frac{3}{4},\frac{3}{4},\frac{8}{9}}&
   \bbn{ \frac{1 - 2{\lambda }^2}{\lambda  - {\lambda }^3}}
    {\frac{4 + 27{\lambda }^2}
     {144{\lambda }^2\left( -1 + {\lambda }^2 \right) }}
    {\frac{32 + 49{\lambda }^2 + 27{\lambda }^4}
     {36{\lambda }^2
       {\left( -1 + {\lambda }^2 \right) }^2}}&6\\\hline
  10& \abn{ I_1 ,I_2 ,I_5 ,IV }
  { \frac{-27}{4} ,\frac{-1}{2} ,\infty,0}
   { \frac{3}{4},\frac{3}{4},\frac{3}{4},\frac{8}{9}}&
   \bbn{ \frac{27 + 87\lambda  + 16{\lambda }^2}
     {27\lambda  + 58{\lambda }^2 + 8{\lambda }^3}}
    {\frac{-3 + 16\lambda  + 6{\lambda }^2}
     {4{\lambda }^2\left( 27 + 58\lambda  +
         8{\lambda }^2 \right) }}
    {\frac{648 + 1824\lambda  + 3157{\lambda }^2 +
       476{\lambda }^3 + 48{\lambda }^4}{{\lambda }^2
       {\left( 27 + 58\lambda  + 8{\lambda }^2 \right)
           }^2}}&10\\\hline
  11& \abn{ I_3 ,I_3 ,I_2 ,IV }{ \infty,0,-1,1}
{ \frac{3}{4},\frac{3}{4},\frac{3}{4},\frac{8}{9}}&
   \bbn{ \frac{-1 + \lambda  + 4{\lambda }^2}
     {2\left( -\lambda  + {\lambda }^3 \right) }}
    {\frac{-13 - 22\lambda  + 27{\lambda }^2}
     {144{\left( -1 + \lambda  \right) }^2
       \left( \lambda  + {\lambda }^2 \right) }}
    {\frac{27 + 5\lambda  + 64{\lambda }^2 +
       5{\lambda }^3 + 27{\lambda }^4}{36
       {\lambda }^2{\left( -1 + {\lambda }^2 \right) }^2}}&6

\end{array}
\]
\caption{Twists of elliptic surfaces}\label{bigtablenew}
\end{table}

\section{Schwarzian derivatives}
\label{sec:schwartzian}

Our goal in the rest of this work is to compare the differential
equations that we obtained in the previous section to those obtained
by Elkies in~\cite{Elk98}. One essential problem is that the equation
depends in an essential way on the choice of the section of the de
Rham bundle. Even, as is the
case for us, if the choice is between different sections of a line
bundle, it still means that the periods could be multiplied by an
arbitrary rational function. To compare two differential equations it
is best to compare quantities which are invariant with respect to such
scaling.

This can be done as follows for equations of degree 2. Consider the
quotient of two independent solutions. This is invariant with respet
to scaling. It depends of course on the choice of the two solutions,
but only up to a fractional linear transformation. Applying the
Schwarzian derivative, to be recalled next, removes this
ambiguity. Our reference for this material is~\cite{Elk98} (see
also~\cite{Iha74}).

\begin{definition}
  The Schwarzian derivative of a function $z$ with respect to the
  parameter $\zeta$ is the function~\cite[(13)]{Elk98}
  \begin{equation*}
    S_\zeta(z) = \frac{2z' z''' - 3 (z'')^2}{(z')^2}
  \end{equation*}
  where derivatives are with respect to $\zeta$.
\end{definition}
We recall the following relevant results

\begin{proposition}
  \begin{enumerate}
  \item If $z_1$ is obtained from $z$ by a fractional linear
    transformation, then $ S_\zeta(z_1)= S_\zeta(z) $.
  \item If $z$ is the quotient of a basis of solutions to the
    differential equation $y''+ay'+by=0$, derivative taken with
    respect to $\zeta$, then the Schwarzian derivative of $z$, which
    is independent of the choice of solutions by the first part, is
    given by~\cite[(17)]{Elk98}
    \begin{equation*}
      S_\zeta(z) = 4b - a^2-2 a'\;.
    \end{equation*}
  \end{enumerate}
\end{proposition}
This gives our required invariant.

To describe the dependency of the parameter $\zeta$, it is better to
consider the quadratic differential
\begin{equation*}
  \sigma_\zeta(z) = S_\zeta(z) (d\zeta)^2
\end{equation*}
We have the formula~\cite[(14)]{Elk98}
\begin{equation*}
  S_\eta(z) = \left(\frac{d\zeta}{d\eta}\right)^2 S_\zeta(z) + S_\eta(\zeta)
\end{equation*}
and thus
\begin{equation*}
  \sigma_\eta(z) = \sigma_\zeta(z) + \sigma_\eta(\zeta)
\end{equation*}
so the quadratic differential $\sigma$ is not independent of the
parameter, but has a simple transformation formula with respect to
changing the variable. This will allow us, using the sigma invariant,
to determine the change of variables that will take one equation to
another.

To allow us to guess the required change of variables more easily, we
further study the residue of the sigma invariant and how it behaves
with respect to change of variables.
 
An honest quadratic differential $f(\zeta) (d\zeta)^2$ has a well defined
residue which is the coefficient of $\zeta^{-2}$ in $f$. with a chance
of variable $\zeta=\zeta(\eta)$ which has order $n$ the residue is
multiplied by a factor of $n^2$.

Suppose now that $\zeta= \eta^n$. Then
\begin{align*}
  S_\eta(\zeta)&= 2 \frac{n(n-1)(n-2)\eta^{n-3}}{n\eta^{n-1}} -
  3\left(\frac{n(n-1)\eta^{n-2}}{n\eta^{n-1}}\right)^2\\
  &= \eta^{-2} \left(2(n-1)(n-2)-3 (n-1)^2\right) \\
  &= \eta^{-2}\left(2(n^2-3n+2)-3(n^2-2n+1)\right)\\
  &= \eta^{-2} (1-n^2)
\end{align*}
It follows that if the coefficient of $\zeta^{-2}$ in $S_\zeta(z)$ is
$\alpha$ then the coefficient of $\eta^{-2}$ in $S_\zeta(z)$ is
\begin{equation*}
  n^2 \alpha + (1-n^2)= 1- n^2(1-\alpha)
\end{equation*}
This leads to the following.
\begin{definition}
  The schwarzian residue of $\sigma=f(\zeta)(d\zeta)^2$, denoted
  $\res_S \sigma$,  is $1-$ the coefficient of
  $\zeta^{-2}$ in $f(\zeta)$.
\end{definition}
and we have
\begin{proposition}
  if $\zeta= \zeta(\eta)$ is a change of variables of degree $n$ then
  \begin{equation*}
    \res_S \sigma_\eta(z)= n^2 \res_S \sigma_\zeta(z)
  \end{equation*}
\end{proposition}

Note that at points where the differential is holomorphic the
Schwarzian residue is $1$ and not $0$. In our examples, the local
system pulls back to the Shimura local system $\Sh$ over the upper
half plane $\upH$, where it is holomorphic. Thus, the Schwarzian
residue is always $1/n^2$ for
some $n$. We call this $n$ the \emph{Schwarzian index} at the
point. It is, of course,
just the ellipticity index of the point.

In the following table we list for the examples we have the sigma
invariant (with respect to the parameter $\lambda$, neglecting the
$d\lambda^2$ term, and the Schwarzian indexes at points when it is
bigger than $1$ (where the original fibration had a singular fiber).

\begin{table}[htbp]
\[
\begin{array}{||l||l|l|l||}

  1& \cbn{ I_1 ,I_1 ,I_8 ,II }
  {\gamma,\bar{\gamma},\infty,0}
    { 2,2,2,6}&
    \dbn{ \frac{27 - 21\lambda  +
       6{\lambda }^2}{27\lambda  - 14{\lambda }^2 +
       3{\lambda }^3}}
    {\frac{3\left( -1 - 6\lambda  + 3{\lambda }^2
         \right) }{16{\lambda }^2
       \left( 27 - 14\lambda  + 3{\lambda }^2 \right) }}
    {\frac{3\left( 945 - 652\lambda  +
         142{\lambda }^2 - 60{\lambda }^3 +
         9{\lambda }^4 \right) }{4{\lambda }^2
       {\left( 27 - 14\lambda  + 3{\lambda }^2 \right)
           }^2}}\\\hline
  2& \cbn{ I_1 ,I_2 ,I_7 ,II }
  {\frac{-9}{4},\frac{-8}{9},\infty,0}
   { 2,2,2,6}&
   \dbn{ \frac{144 + 339\lambda  + 144{\lambda }^2}
     {144\lambda  + 226{\lambda }^2 + 72{\lambda }^3}}
    {\frac{-2 + 36\lambda  + 27{\lambda }^2}
     {4{\lambda }^2\left( 72 + 113\lambda  +
         36{\lambda }^2 \right) }}
    {\frac{20160 + 42008\lambda  + 41331{\lambda }^2 +
       17388{\lambda }^3 + 3888{\lambda }^4}{4
       {\lambda }^2{\left( 72 + 113\lambda  +
           36{\lambda }^2 \right) }^2}}\\\hline
  3& \cbn{ I_1 ,I_4 ,I_5 ,II }{ -10,0,\infty,\frac{1}{8}}
   { 2,2,2,6}&
   \dbn{ \frac{-5 + 119\lambda  + 16{\lambda }^2}
     {\lambda \left( -10 + 79\lambda  +
         8{\lambda }^2 \right) }}
    {\frac{6\left( -1 + 7\lambda  + 2{\lambda }^2
         \right) }{{\left( 1 - 8\lambda  \right) }^2
       \lambda \left( 10 + \lambda  \right) }}
    {\frac{3\left( 25 - 210\lambda  +
         2179{\lambda }^2 + 216{\lambda }^3 +
         16{\lambda }^4 \right) }{{\left( 1 -
           8\lambda  \right) }^2{\lambda }^2
       {\left( 10 + \lambda  \right) }^2}}\\\hline
  4& \cbn{ I_2 ,I_3 ,I_5 ,II }{\frac{-5}{9},0,\infty,3}
   { 2,2,2,6}&
   \dbn{ \frac{15 + 39\lambda  - 36{\lambda }^2}
     {30\lambda  + 44{\lambda }^2 - 18{\lambda }^3}}
    {\frac{-23 - 246\lambda  + 81{\lambda }^2}
     {48{\left( -3 + \lambda  \right) }^2\lambda
       \left( 5 + 9\lambda  \right) }}
    {\frac{2025 + 4295\lambda  + 9156{\lambda }^2 -
       1809{\lambda }^3 + 729{\lambda }^4}{12
       {\left( -3 + \lambda  \right) }^2{\lambda }^2
       {\left( 5 + 9\lambda  \right) }^2}}\\\hline
  5& \cbn{ I_1 ,I_1 ,I_7 ,III }
  {\delta,\bar{\delta},\infty,0}
  { 2,2,2,4}&
   \dbn{ \frac{64 + 39\lambda  + 16{\lambda }^2}
     {64\lambda  + 26{\lambda }^2 + 8{\lambda }^3}}
    {\frac{-2 + 4\lambda  + 3{\lambda }^2}
     {4{\lambda }^2\left( 32 + 13\lambda  +
         4{\lambda }^2 \right) }}
    {\frac{3840 + 2072\lambda  + 43{\lambda }^2 +
       220{\lambda }^3 + 48{\lambda }^4}{4
       {\lambda }^2{\left( 32 + 13\lambda  +
           4{\lambda }^2 \right) }^2}}\\\hline
  6& \cbn{ I_1 ,I_2 ,I_6 ,III }{ 4,1,\infty,0}
  { 2,2,2,4}&
   \dbn{ \frac{8 - 15\lambda  + 4{\lambda }^2}
     {2\lambda \left( 4 - 5\lambda  +
         {\lambda }^2 \right) }}
    {\frac{-1 - 6\lambda  + 3{\lambda }^2}
     {16{\lambda }^2
       \left( 4 - 5\lambda  + {\lambda }^2 \right) }}
    {\frac{3\left( 20 - 33\lambda  + 28{\lambda }^2 -
         7{\lambda }^3 + {\lambda }^4 \right) }{4
       {\lambda }^2{\left( 4 - 5\lambda  +
           {\lambda }^2 \right) }^2}}\\\hline
  7& \cbn{ I_1 ,I_3 ,I_5 ,III }
  {\frac{-25}{3},0,\infty,\frac{1}{5}}
   { 2,2,2,4}&
   \dbn{ \frac{25 - 369\lambda  - 60{\lambda }^2}
     {50\lambda  - 244{\lambda }^2 - 30{\lambda }^3}}
    {\frac{-167 + 630\lambda  + 225{\lambda }^2}
     {16{\left( 1 - 5\lambda  \right) }^2\lambda
       \left( 25 + 3\lambda  \right) }}
    {\frac{15\left( 125 - 675\lambda  +
         4244{\lambda }^2 + 501{\lambda }^3 +
         45{\lambda }^4 \right) }{4
       {\left( 1 - 5\lambda  \right) }^2{\lambda }^2
       {\left( 25 + 3\lambda  \right) }^2}}\\\hline
  8& \cbn{ I_2 ,I_3 ,I_4 ,III }{ \frac{-1}{3} ,0,\infty,1}
   { 2,2,2,4}&
   \dbn{ \frac{1 + 3\lambda  - 12{\lambda }^2}
     {2\lambda  + 4{\lambda }^2 - 6{\lambda }^3}}
    {\frac{-1 - 9\lambda  + 9{\lambda }^2}
     {16{\left( -1 + \lambda  \right) }^2
       \left( \lambda  + 3{\lambda }^2 \right) }}
    {\frac{3\left( 1 + 3\lambda  + 13{\lambda }^2 -
         6{\lambda }^3 + 9{\lambda }^4 \right) }{4
       {\left( -1 + \lambda  \right) }^2
       {\left( \lambda  + 3{\lambda }^2 \right) }^2}}
       \\\hline
   9& \cbn{ I_1 ,I_1 ,I_6 ,IV }{ 1,-1,\infty,0}
   { 2,2,2,3}&
   \dbn{ \frac{1 - 2{\lambda }^2}{\lambda  - {\lambda }^3}}
    {\frac{4 + 27{\lambda }^2}
     {144{\lambda }^2\left( -1 + {\lambda }^2 \right) }}
    {\frac{32 + 49{\lambda }^2 + 27{\lambda }^4}
     {36{\lambda }^2
       {\left( -1 + {\lambda }^2 \right) }^2}}\\\hline
  10& \cbn{ I_1 ,I_2 ,I_5 ,IV }
  { \frac{-27}{4} ,\frac{-1}{2} ,\infty,0}
   { 2,2,2,3}&
   \dbn{ \frac{27 + 87\lambda  + 16{\lambda }^2}
     {27\lambda  + 58{\lambda }^2 + 8{\lambda }^3}}
    {\frac{-3 + 16\lambda  + 6{\lambda }^2}
     {4{\lambda }^2\left( 27 + 58\lambda  +
         8{\lambda }^2 \right) }}
    {\frac{648 + 1824\lambda  + 3157{\lambda }^2 +
       476{\lambda }^3 + 48{\lambda }^4}{{\lambda }^2
       {\left( 27 + 58\lambda  + 8{\lambda }^2 \right)
           }^2}}\\\hline
  11& \cbn{ I_3 ,I_3 ,I_2 ,IV }{ \infty,0,-1,1}
{ 2,2,2,3}&
   \dbn{ \frac{-1 + \lambda  + 4{\lambda }^2}
     {2\left( -\lambda  + {\lambda }^3 \right) }}
    {\frac{-13 - 22\lambda  + 27{\lambda }^2}
     {144{\left( -1 + \lambda  \right) }^2
       \left( \lambda  + {\lambda }^2 \right) }}
    {\frac{27 + 5\lambda  + 64{\lambda }^2 +
       5{\lambda }^3 + 27{\lambda }^4}{36
       {\lambda }^2{\left( -1 + {\lambda }^2 \right) }^2}}

\end{array}
\]
\caption{Sigma invariants and Schwarzian indices}\label{sigmatable}
\end{table}

\section{The results of Elkies}
\label{sec:elkies}

In~\cite{Elk98} Elkies computes certain differential equations
associated with Shimura curves. While this is not stated explicitly,
these are exactly the Picard Fuchs equations associated with the
Shimura local system descended to the Shimura curve
because the quotient of the two solutions gives the coordinate $\tau$
on the upper half plane, just as for the Shimura local system, as in
Section~\ref{sec:generalities}.

The briefly list the types of Shimura curves considered. For more
information one may consult~\cite{Elk98} or~\cite{Bes-Liv10} (our
notation is consistent with the latter reference). For each
discriminant $D$ (always the product of an even number of primes) the
Shimura curve $V_D$ is the quotient of the upper half plane by the
group $\Gamma$ of norm one elements in a maximal order in a quaternion
algebra of discriminant $D$ (see the introduction). For each prime
$p|D$ it carries a modular involution $w_p$ and these involutions commute
with each other so that we can also set $w_n= \prod_{p|n} w_p$ for
$n|D$. We let $\vstar{D}$ be the quotient of $V_D$ by the group
generated by all its modular involutions. Finally, for a prime $p$
which does not divide $D$ there is a modular curve $V_{D,p}$, which
corresponds to an additional ``$\Gamma_0(p)$'' structure, This retains
all of the previous involutions but has an additional involution
$w_p$.

In table~\ref{elkies} we give, for the relevant curves, the
equation that Elkies finds, the associated sigma invariant and the
Schwarzian indices at the relevant points. The
equations of Elkies are in non-normalized form $ay''+by'+cy=0$, so one
needs to normalize first before computing the sigma invariant.
\begin{table}[htbp]
\[
\begin{array}{||l||l|l|l||}

  \vstar{10}
&
\ebn{ 
t (t - 2) (t - 27)y''+  \frac{10 t^2 - 203 t +
216}{6} y'+(\frac{7 t}{144} - \frac{7}{18})y=0}
{\frac{10368 - 7296\,t + 3157\,t^2 - 119\,t^3 +
3\,t^4} {4\,{\left( -27 + t \right) }^2\,
      {\left( -2 + t \right) }^2\,t^2}
}

&
 \cbn{ I_1 ,I_1 ,I_8 ,II }
  {27,2,\infty,0}
    { 2,2,2,3}
\\\hline

\vstar{14}

&
\ebn{
t (16 t^2 + 13 t + 8)y''+ (24 t^2 + 13 t + 4)y'+
(\frac{3}{4} t + \frac{3}{16}) y=0}
{
\frac{192 + 440\,t + 43\,t^2 + 1036\,t^3 +
960\,t^4} {4\,t^2\,{\left( 8 + 13\,t + 16\,t^2
\right) }^2}
}

&
 \cbn{ I_1 ,I_1 ,I_8 ,II }
  {\delta_1,\bar{\delta}_1,0,\infty}
    { 2,2,2,4}
\\\hline

\vstar{15}

&
\ebn{
(t-81) (t-1) t y''+ (\frac{3 t^2}{2}-82 t+\frac{81}{2})y'+
(\frac{t}{18}-\frac{1}{2}) y=0}
{
\frac{35 t^4-3680 t^3+244242 t^2-244944 t+177147}{36 (t-81)^2 (t-1)^2 t^2}
}

&
 \cbn{ I_1 ,I_1 ,I_8 ,II }
  {1,81,0,\infty}
    { 2,2,2,6}
\\\hline
\end{array}
\]
 \caption{Elkies's list of differential equations}\label{elkies}
\end{table}

Here $\delta_1$ is a solution to the equation $16t^2+13t+8=0$.

For discriminant 6 Elkies does not write down the equation
explicitely, though he gives a recepy to discover one of 4 possible
equations. As he notes, however, since there are only 3 elliptic
points, the sigma invariant is uniquely determined by the indexes
of ellipticity. Suppose that these are at $t=0,1,\infty$. The most
general form of $\sigma$ is 
\begin{equation*}
  \sigma =
  \left(\frac{a}{t^2}+\frac{b}{(t-1)^2}+\frac{c}{t}+\frac{d}{t-1}\right) (dt)^2
\end{equation*}
and one has the condition $c+d=0$ to avoid a pole of order $3$ at
$\infty$. The residues are $a,b$ and $a+b+d$ at $0$, $1$ and
$\infty$ respectively, from which all the coefficients are easily
determined. In the case at hand, Elkies chooses the coordinate $t$ so
that the indices are $2,4,6$ at  $0$, $1$ and
$\infty$ respectively. This gives
\begin{equation*}
  \sigma =
  \left(\frac{3}{4t^2}+\frac{15}{16(t-1)^2}
    +\frac{103}{144t}-\frac{103}{144(t-1)}\right)
  (dt)^2 
\end{equation*}

\section{comparison with the results of elkies}
\label{sec:comparison}
In this section we compare Elkies's list with the list of differential
equations we obtained in Section~\ref{sec:k3}. As explained in the
introduction, each of these examples is a family of varieties over
$\PP^1$ and there is a correspondence between these $\PP^1$ and a
Shimura curve of some (computable) discriminant.

The above correspondece is compatible with the Shimura local system
$\Sh$. This implies that the correspondence is going to map the sigma
invariants on Elkies's list to the corresponding sigma invariant of
the families. Here we demonstrate how one can use this to guess the
correct correspondence. This is not a proof that the correspondence is
the correct one, which then needs to be established by more precise
means~\cite[Section~8]{Bes-Liv10}. It can nevertheless
be used to exclude
certain possible correspondenced (see Subsection~8.2 in the
above reference).

No. 10  - Corresponds to discriminant $10$. The correspondence has to
carry the special points of the
fibration at $\lambda=-27/4,-1/2,\infty,0 $ with respective Schwarzian
indices $2,2,2,3$ to the special points $t=27,2,\infty,0 $ with the
same respective indices for the equation that Elkies finds for
$V_{10}/(w_5,w_2)$. It is trivial to guess the relation
$t=-4\lambda$. A change of variables for the sigma invariants confirms
this, It can be proved rigorously
(see~\cite[Subsection~8.3]{Bes-Liv10}) that the $\lambda$-line
is isomorphic to $V_{14}/(w_2,w_5)$

No. 5  - Corresponds to discriminant $14$. The correspondence has to
carry the special points of the
fibration at $\lambda= \delta,\bar{\delta},\infty,0 $ with respective Schwarzian
indices $2,2,2,4$ to the special points
$t=\delta_1,\bar{\delta}_1,0,\infty $ with the
same respective indices for the equation that Elkies finds for
$V_{14}/(w_7,w_2)$. Since $\delta$ is a solution of the equation
$4x^2+13x+32=0$  it is easy to guess the relation
$t=2/\lambda$. A change of variables for the sigma invariants confirms
this. It can be proved rigorously
(see~\cite[Subsection~8.1]{Bes-Liv10}) that the $\lambda$-line
is indeed isomorphic to the Shimura curve $V_{14}/(w_2,w_7)$.

No. 3 - Corresponds to discriminant $15$.
 The correspondence has to
carry the special points $t=1,81,0,\infty $  with respective Schwarzian
indices $2,2,2,6$  for the equation that Elkies finds for
$V_{15}/(w_5,w_3)$ to the special points of the 
fibration at $\lambda=-10,0,\infty,1/8$ with the same indices. This
can be done with the change of variables $\lambda=\frac{t-81}{8t}$ and
a change of variables for the sigma invariants confirms
this. It can be proved rigorously (see~\cite[Lemma~3.8.2]{Bes-Liv10}).

No. 4 - Corresponds to discriminant $10$. In this case we speculated
(but could not prove) that the $\lambda$-line was the curve
$V_{10,3}/\pair{w_2,w_5,w_3} $. Here we show this is consistent with
the Picard-Fuchs equations. According to Elkies, the curve
$V_{10,3}/\pair{w_2,w_5} $ is a degree $4$ cover of $\vstar{10}$,
rational with a coordinate $x$ such that
\begin{equation*}
  t=\frac{{\left( -6 + 6\,x
\right) }^3} {{\left( 1 + x \right) }^2\, \left(
17 - 10\,x + 9\,x^2 \right) }
\end{equation*}
(this is equation (57) in~\cite{Elk98} but the $7$ there should be
corrected to $17$, as for example in the computation between equations
(59) and (60)). From the expression
\begin{equation*}
  \frac{6^3}{9\tau + 8} \text{ with } \tau= \frac{(3x^2+5)^2}{9(x-1)^3}
\end{equation*}
which is also in (57) there it is easy to see that the map $x\to t$
sends $1,\infty,-1, 5$ to $0,0,\infty,2$ 
with multiplicities $3,1,2, 2$ respectively, $\pm \sqrt{-5/3}$ to
$27$ with multiplicity $2$, the two roots of $9x^2-10x+5=0$ to $2$
with multiplicity $1$, and the two roots of $9x^2-10x+17=0$ to $\infty$
with multiplicity $1$. Thus, the elliptic points for
$V_{10,3}/\pair{w_2,w_5} $ are going to
be at $x=\infty$ with multiplicity $3$ and at the roots of the
equations  $9x^2-10x+5=0$ and $9x^2-10x+17=0$ with multiplicity
$2$. The involution $w_3$ is given by Elkies, just after (57), to be
$w_3(x)= \frac{10}{9}-x$ and so a coordinate on the quotient is given
by 
$$\zeta= 9\left(x-\frac{5}{9}\right)^2 = 9x^2-10x+ \frac{25}{9}$$
We see that the elliptic points will map to
$\zeta= \infty, -20/9, -128/9$, so these will be elliptic of degree
$6,2,2$, and in addition the ramification point $0$ is elliptic of
degree $2$. We can map $\zeta$ to $\lambda$ with the correct orders by
\[ \lambda = 3- \frac{128}{3\zeta + \frac{128}{3}} =  3-
\frac{128}{3(9x^2-10x+17)}\;. \]
This is confirmed by the matching of the sigma invariants.

Other examples correspond to discriminant $6$. Some of them are
directly interrelated.
 Consider examples number 6 and 8. The special points are
 $\lambda_1=4,1,\infty,0$ and
$\lambda_2=-1/3,0,\infty,1$ respectively with the same indices. There
is a finite number of ways for carying one set to the other preserving
the indices, and testing each one using the sigma invariants we get
the correct transformation  $\lambda_1= 1 -
1/\lambda_2$. It
turns out that (see~\cite[Subsection~8.7]{Bes-Liv10}) that
making this change of variable makes the two base elliptic fibrations
isogenous.
 
Consider next examples 9 and 11.
The special points in both cases are $\infty,0,-1,1$ but with indices
$2,3,2,2$ in example 9 and $2,2,2,3$ in example 11. Testing again the
finite number of possible transformations with the sigma invariants
gives $\lambda_2=(1 + \lambda_1)/(1 - \lambda_1)$. It is proved
in~\cite[Subsection~8.8]{Bes-Liv10} that this again makes the
two base fibrations isogenous.

No. 6 - For $V_6/(w_2,w_3)$ it turns out to be slightly better to work
with the coordinate $\zeta = 1/(1-t)$ so that the elliptic points are
at $\zeta= 0$, $1$ and
$\infty$ with indices $6$, $2$ and $4$
respectively. To get the required ellipticity behavior for the
$\lambda$-line, with elliptic points at $\lambda=4,1,\infty,0$ with
indices $2,2,2,4$, one may consider a degree 3 map having ramification
type $(2,1)$ over $\infty$, producing indices $2$
and $4$, ramification $3$ above $0$, producing an
additional index $2$ and ramification type
$(1,2)$ above $1$, producing one additional index
$2$ and an additional non-elliptic point. This can be aranged by a map
of the form $\zeta = c\lambda^{-1} ((\lambda-1)^3$ for the appropriate
$c$ for which this ramifies above $1$. So $c$ is the value for which
one of the roots of the derivative $(c(\lambda-1)^3-\lambda)'=
3c(\lambda-1)^2-1 $ is mapped to $\zeta=1$. We have for that root
\begin{equation*}
  1 = c\frac{(\lambda-1)^3}{\lambda} = \frac{\lambda-1}{3\lambda}
\end{equation*}
hence $\lambda = -1/2$ and $c=4/27$. Consider the equation for
$\lambda$ to map to $\zeta=1$. As an equation on $\lambda-1$ the sum
of the 3 roots should be $0$, hence the third root should be $3$, so
that the additional preimage of $1$ is $4$. Thus, the cover we found
matches perfectly with the $\lambda$-line. Summarizing, we have
\begin{equation*}
  t= 1-\frac{1}{\zeta} = 1-\frac{27 \lambda}{4(\lambda-1)^3}
\end{equation*}
This is confirmed by the sigma invariants.

No. 9 - Here the elliptic points are at $1,-1,\infty$ and $0$ with
indices $2,2,2$ and $3$. In trying to relate them with the elliptic
points for $\zeta$ it is very easy to guess the relation
$\zeta=\lambda^2$, and this is confirmed by the
$\sigma$-invariants. Thus, the $\lambda$-line is a double cover of
$\vstar{6}$ ramified above the elliptic points of order $4$ and
$6$. This was used in~\cite[Subsection 8.2]{Bes-Liv10} to prove that
the $\lambda$-line is 
$V_6/\pair{w_6}$.

No. 7 - The elliptic points are $-25/3,0,\infty,1/5$
and with indices $2$ at the first $3$ points and
$4$ at the last point.
Here we guess that the base for the family of
twists is isomorphic to the quotient
$V_{6,5}/\pair{w_2,w_3,w_5}$. Elkies find a coordinate $x$ on
$V_{6,5}/\pair{w_2,w_3}$ for which the action
of $w_5$ is given by \cite[(37)]{Elk98}
$w_5(x)=(42 - 55 x)/(55 + 300 x)$. The two fixed
points of this action are $7/30$ and $-3/5$.
Thus, a coordinate on the quotient is provided by
\begin{equation}\label{dsix1}
  y=((x + 3/5)/(x - 7/30))^2\;.
\end{equation}
The map from $X_0^\ast(5)$ to $V_6/(w_2,w_3)$ is
given by \cite[Equation 36]{Elk98} by
\begin{equation*}
  t=( 1 + 3x + 6x^2 )^2
  ( 1 - 6x + 15x^2)=1+27x^4 ( 5 + 12x + 20x^2 )\;.
\end{equation*}
The relation between $t$ and $\zeta$ is
\begin{equation}\label{dsix2}
\zeta= 1/(1-t)=\frac{-1}{27x^4 ( 5 + 12x +
20x^2)}\;.
\end{equation}
The ramification above $\zeta=0$ is of order $6$ at
infinity. The ramification over $\zeta=\infty$ is
of order $4$ at $x=0$ and order $1$ at each of the
roots of $5 + 12x + 20x^2$. The ramification over
$\zeta=1$, or $t=0$, is of order $1$ at each of the
roots of $1 - 6x + 15x^2$ and of order $2$ at each
of the roots of $ 1 + 3x + 6x^2$. Thus the elliptic
points of the cover are of order $4$ at the roots
of $ 5 + 12x + 20x^2$ and of order $2$ at  each of
the roots of $1 - 6x + 15x^2$. These two pairs of
points are interchanged by $w_5$. The elliptic
points of order $4$ are mapped to $y=-9/16$ while
those of order $2$ are mapped to $y=-24$. In
addition we get elliptic points at the ramification
points of the covering at $y=0$ and $y=\infty$,
both of order $2$. Now, if we guessed correctly,
there would be a M\"obius transformation sending
the 4 elliptic points to the 4 singular points of
the elliptic surface, sending  $y=-9/16$ to
$\lambda=1/5$ . It is easy to check that the unique
transformation of this type is $\lambda=-25 y/(3
(24+y))$. Composing with \eqref{dsix1} we get
\begin{equation*}
\lambda=\frac{-( 3 + 5x)^2}{5 ( 1 - 6x +
15x^2)}\;.
\end{equation*}
This, as usual, is confirmed by Pulling back the
$\sigma$-invariants. Our guess can be proved
rigorously~\cite[Subsection~8.6]{Bes-Liv10}.

No. 2 - The elliptic points are at $-9/4,-8/9,\infty,0$ with indices
$2,2,2,6$ respectively.

We try to guess that this family corresponds to
$V_{6,7}/\pair{w_2,w_3,w_7}$. Elkies write a coordinate $x$ on
$V_{6,7}/\pair{w_2,w_3}$ for which  the action of $w_7$ is given by
\cite[(40)]{Elk98} $w_7(x)=(116 - 9 x)/(9
+ 20x)$. The two fixed points of this action are
$-29/10$ and $2$. Thus, a coordinate on the
quotient is provided by
\begin{equation}\label{dsix4}
  y=\left(\frac{x + 29/10}{x - 2}\right)^2\;.
\end{equation}
The map from $V_{6,7}/\pair{w_2,w_3}$ to
$\vstar{6}$ is given by \cite
[(39)]{Elk98} by
\begin{equation*}
t= \frac{- \left( 25 + 4\,x + 4\,x^2 \right) \,
{\left( 2 - 12\,x + 3\,x^2 - 2\,x^3 \right) }^2
}{108\,\left( 37 - 8\,x + 7\,x^2 \right) }=
1-\frac{(2 x^2 - x + 8)^4}{108 (7 x^2 - 8 x +
37)}\;.
\end{equation*}
This looks nicer with $\zeta$
\begin{equation}\label{dsix5}
\zeta= 1/(1-t)=\frac{108\,\left( 37 - 8\,x +
7\,x^2 \right) }
  {{\left( 8 - x + 2\,x^2 \right) }^4}\;.
\end{equation}
The preimage of  $\zeta=0$ is $6$ times
$\infty$ plus the two roots of $7
x^2 - 8 x + 37$. The preimage of
$\zeta=\infty$ is $4$ times each of the
roots of $2 x^2 - x + 8$. The preimage of
$\zeta=1$, or $t=0$, is 2 times  each of
the roots of $ 2 - 12x + 3x^2 - 2x^3$ plus each of the
roots of $4x^2+4x+25$. Thus the
elliptic points of $V_{6,7}/\pair{w_2,w_3}$ are
the two roots of $7 x^2 - 8 x + 37$ with index $6$ and the roots of
$4x^2+4x+25$ with index $2$.
two pairs of points are interchanged by $w_7$. The
elliptic points of order $6$ are mapped to
$y=-243/100$ while those of order $2$ are mapped to
$y=-24/25$. In addition we get elliptic points at
the ramification points of the covering at $y=0$
and $y=\infty$, both of order $2$. Now, if we
guessed correctly, there would be a M\"obius
transformation sending the 4 elliptic points to the
4 singular points of the elliptic surface, sending
$y=-243/100$ to $\lambda=0$ . It is easy to check
that the unique transformation of this type is
\begin{equation*}
  \lambda=\frac{486 + 200y}{-216 - 225y}\;.
\end{equation*}
Composing with \eqref{dsix4} we get
\begin{equation*}
\lambda=\frac{-8\,\left( 37 - 8\,x + 7\,x^2
\right) }
  {9\,\left( 25 + 4\,x + 4\,x^2 \right) }\;.
\end{equation*}
This is confirmed by the $\sigma$-invariants. Our guess can be proved
rigorously~\cite[Subsection~8.5]{Bes-Liv10}.

\section{Software}
\label{sec:software}

All the relevant computations for this work can be downloaded from
\url{http://www.math.bgu.ac.il/~bessera/picard-fuchs/}. They are in the form of a MATHEMATICA notebook. The notebook
is self explanatory. It loads the following files:
\begin{itemize}
\item pf.m - main file containing all the algorithms
\item data.m - file contains the equations for the elliptic fibrations
  on Herfurtner's list
\item elkiesdata.m - contains the differential equations obtained by Elkies.
\end{itemize}

The relevant functions contained in the file pf.m
\begin{itemize}
\item picfucs function - computes the Picard-Fuchs equation for an
  elliptic surface. This is just a translation into Mathematica of the
  Maple script by Beukers, which may be found at
\url{http://www.staff.science.uu.nl/~beuke106/picfuchs.maple}
\item Twistpf function - computes the Picard-Fuchs equation for the
  family of twists given the equation for the original elliptic fibration.
\item SigChVar function - makes a change of variable for the sigma invariant.
\item DRes function - computes the residue of a quadratic differential.
\end{itemize}

\end{document}